\documentclass[12pt]{amsart}
\usepackage{a4wide}
\usepackage{amsmath}
\usepackage{amssymb}
\usepackage{mathtools}
\usepackage{graphicx} 
\usepackage{color} 
\usepackage{indentfirst} 
\usepackage{mathrsfs} 
\usepackage{bm} 
\usepackage{bbm} 
\usepackage{yhmath} 
\usepackage[all]{xy} 
\usepackage[colorlinks,linkcolor=blue,anchorcolor=blue,citecolor=blue]{hyperref}
\usepackage{cite}

\newcommand{\mcal}{\mathcal}

\newcommand{\set}[1]{\left\{ #1 \right\}}

\newcommand{\C}{\mathbb{C}}
\newcommand{\R}{\mathbb{R}}

\newcommand{\Z}{\mathbb{Z}}
\newcommand{\N}{\mathbb{N}}

\newcommand{\f}{\infty}

\newcommand{\wh}[1]{\widehat{#1}}
\newcommand{\wt}[1]{\widetilde{#1}}

\newcommand{\ep}{\varepsilon}
\newcommand{\sse}{\subseteq}
\newcommand{\sm}{\setminus}
\newcommand{\D}{\;\mathrm{d}}

\newcommand{\cv}[2]{
\left(
  \begin{array}{c}
    #1 \\
    #2 \\
  \end{array}
\right)
}

\usepackage{amsthm}
\newtheorem{theorem}{Theorem}[section]

\newtheorem{lemma}[theorem]{Lemma}
\newtheorem{definition}[theorem]{Definition}
\newtheorem{corollary}[theorem]{Corollary}
{
  \theoremstyle{definition}
  \newtheorem{example}[theorem]{Example}
}
\theoremstyle{remark}
\newtheorem*{remark}{Remark}

\numberwithin{equation}{section}

\allowdisplaybreaks[4]

\usepackage{fancyhdr}
\title{The spectrality of infinite convolutions in $\mathbb{R}^d$}

\author[W. Li]{Wenxia Li}
\address[W. Li]{School of Mathematical Sciences, Key Laboratory of MEA (Ministry of Education) \& Shanghai Key Laboratory of PMMP, East China Normal University, Shanghai 200241, People's Republic of China}
\email{wxli@math.ecnu.edu.cn}

\author[Z. Wang]{Zhiqiang Wang*}
\address[Z. Wang]{School of Mathematical Sciences, Key Laboratory of MEA (Ministry of Education) \& Shanghai Key Laboratory of PMMP, East China Normal University, Shanghai 200241, People's Republic of China}
\email{zhiqiangwzy@163.com}

\thanks{* Corresponding author}

\begin{document}

\begin{abstract}
In this paper, we study the spectrality of infinite convolutions in $\mathbb{R}^d$, where the spectrality means the corresponding square integrable function space admits a family of exponential functions as an orthonormal basis. Suppose that the infinite convolutions are generated by a sequence of admissible pairs in $\mathbb{R}^d$.
We give two sufficient conditions for their spectrality by using the equi-positivity condition and the integral periodic zero set of Fourier transform.
By applying these results, we show the spectrality of some specific infinite convolutions in $\mathbb{R}^d$.
\end{abstract}


\maketitle

\section{Introduction}

Let $\mcal{P}(\R^d)$ be the collection of Borel probability measures on $\R^d$.  We call $\mu\in \mcal{P}(\R^d)$ a \emph{spectral measure} if there exists a countable subset $\Lambda \sse \R^d$ such that the family of exponential functions
$$\set{ e_\lambda(x) = e^{2\pi i \lambda \cdot x}: \lambda \in \Lambda}$$
forms an orthonormal basis for $L^2(\mu)$, where $\cdot$ denotes the standard inner product on $\R^d$. The set $\Lambda$ is called a \emph{spectrum} of $\mu$. Such orthonormal bases are used for Fourier series expansions of functions \cite{Strichartz-2006}, therefore the existence of spectrum is a fundamental question in harmonic analysis.

In 1974, Fuglede related the existence of commuting self-adjoint partial differential operators to the spectrality and proposed the following well-known spectral set conjecture  in~\cite{Fuglede-1974}.
\begin{quote}
  \emph{A measurable set $\Gamma \sse \R^d$ with positive finite Lebesgue measure is a spectral set, that is, the normalized Lebesgue measure on $\Gamma$ is a spectral measure, if and only if $\Gamma$ tiles $\R^d$ by translations.}
\end{quote}
This conjecture has been refuted by Tao \cite{Tao-2004} and the others \cite{Kolountzakis-Matolcsi-2006b,Matolcsi-2005} in $d$-dimensional spaces with $d\ge 3$, but the study of the connection between spectrality and tiling has raised a great deal of interest, see \cite{Fan-Fan-Liao-Shi-2019,Iosevich-Katz-Tao-2003,Laba-2001,
Lev-Matolcsi-2019}.

Fractal measures usually appear as singular measures with respect to classical Lebesgue measures, we refer the readers to~\cite{Falco03} for details on fractal geometry.
In \cite{Jorgensen-Pedersen-1998}, Jorgensen and Pedersen discovered that the self-similar measure defined by
$$\mu(\;\cdot\;) = \frac{1}{2} \mu(4 \;\cdot\;) + \frac{1}{2} \mu( 4 \;\cdot\; -2)$$
is a spectral measure, but the standard middle-third Cantor measure is not. From then on, the spectrality of  fractal measures has been extensively investigated, and we refer the readers to \cite{An-Dong-He-2022,An-Fu-Lai-2019,An-He-2014,An-He-He-2019,An-Wang-2021,An-He-Tao-2015,Dai-2012,Dai-Fu-Yan-2021,
Dai-He-Lau-2014,Deng-Chen-2021,Dutkay-Han-Sun-2009,Dutkay-Han-Sun-Weber-2011,Dutkay-Haussermann-Lai-2019,Dutkay-Lai-2014,
Fu-Wen-2017,He-Lai-Lau-2013,He-Tang-Wu-2019,Jorgensen-Pedersen-1998,Laba-Wang-2002,Laba-Wang-2006,Li-2009,LMW21,LMW22,
Liu-Dong-Li-2017,Lu-Dong-Zhang-2022,Shi-2021,Strichartz-2000}.
In this paper, we study the spectrality of infinite convolutions in $\R^d$ which may be regarded as a generalization of self-affine measures.

We write $\mcal{M}(\R^d)$ for the collection of all finite nonzero Borel measures on $\R^d$.
Note that $\mcal{P}(\R^d) \subseteq \mcal{M}(\R^d)$.
For $\mu,\nu \in \mcal{M}(\R^d)$, the convolution $\mu *\nu$ is given by $$ \mu*\nu(B) = \int_{\R^d} \nu(B-x) \D \mu(x)= \int_{\R^d} \mu(B-y) \D \nu(y), $$
for every Borel subset $B\sse \R^d$. For a finite subset $A \sse \R^d$, we define the discrete measure
$$\delta_A = \frac{1}{\# A} \sum_{a\in A} \delta_a, $$
where $\#$ denotes the cardinality of a set and $\delta_a$ denotes the Dirac measure concentrated on the point $a$.

Let $R\in M_d(\Z)$ be a $d\times d$ expanding integral matrix, that is, all eigenvalues have modulus strictly greater than $1$, and let $B\sse \Z^d$ be a finite subset of integral vectors.
If there exists a finite subset $L\sse \Z^d$ such that $\#L =\#B$ and the matrix
$$
\left[ \frac{1}{\sqrt{\# B}} e^{-2\pi i (R^{-1}b) \cdot \ell} \right]_{b\in B,\ell \in L}
$$
is unitary, we call $(R, B)$ an {\it admissible pair} in $\R^d$. Sometimes, to emphasize $L$, we call $(R, B, L)$ a  {\it Hadamard triple}.
It is easy to verify that $(R, B)$ is an admissible pair if and only if the discrete measure $\delta_{R^{-1} B}$ admits a spectrum $L \sse \Z^d$.

Given a sequence of admissible pairs $\{(R_n,B_n)\}_{n=1}^\infty$ in $\R^d$, let
\begin{equation}\label{ms_mun}
\mu_n = \delta_{R_1^{-1} B_1} * \delta_{(R_2 R_1)^{-1} B_2} * \cdots * \delta_{(R_n \cdots R_2 R_1)^{-1} B_n}.
\end{equation}
We assume that the weak limit of $\{\mu_n\}$ exists, and the weak limit $\mu$ is called an \emph{infinite convolution}, written as
\begin{equation}\label{infinite-convolution}
  \mu=\delta_{R_1^{-1} B_1} * \delta_{(R_2 R_1)^{-1} B_2} * \cdots * \delta_{(R_n \cdots R_2 R_1)^{-1} B_n} * \cdots.
\end{equation}
Some sufficient and necessary conditions for the existence of infinite convolutions were given in \cite{Li-Miao-Wang-2022}.

The admissible pair assumption implies that all discrete measures $\{\mu_n\}$ are spectral measures.
Therefore, a natural question arises as the following.
\begin{quote}
  \emph{Given a sequence of admissible pairs $\{(R_n,B_n)\}_{n=1}^\f$, under what condition is the infinite convolution $\mu$  a spectral measure ?}
\end{quote}
In fact, it is easy to construct an infinite mutually orthogonal set of exponential functions, but it is very difficult to show the completeness of the orthogonal set. The infinite convolution generated by admissible pairs was first studied by Strichartz \cite{Strichartz-2000} to find more spectral measures.
The admissible pair condition is not enough to guarantee that infinite convolutions are spectral measures, see \cite{An-He-He-2019,LMW21} for counterexamples.
If $(R_n,B_n) =(R,B)$ for all $n \ge 1$, then the infinite convolution $\mu$ is reduced to the self-affine measure.
Dutkay, Haussermann and Lai \cite{Dutkay-Haussermann-Lai-2019} have proved that the self-affine measures generated by an admissible pair is a spectral measure.
We refer the readers to
\cite{An-Fu-Lai-2019,An-He-2014,An-He-He-2019,An-He-Tao-2015,Dai-2012,Dai-Fu-Yan-2021,Dai-He-Lau-2014,Fu-Wen-2017,
Laba-Wang-2002,Li-2009,LMW21,LMW22,Liu-Dong-Li-2017,Lu-Dong-Zhang-2022}
for other related results on spectrality of infinite convolutions.
Currently, most of the research on spectrality of infinite convolutions focus on either the one-dimensional case or some special examples in $\R^d$.

We denote the tail of the infinite convolution $\mu$ by
$$\mu_{>n} = \delta_{(R_{n+1} \cdots R_2 R_1)^{-1} B_{n+1}} * \delta_{(R_{n+2} \cdots R_2 R_1)^{-1} B_{n+2}} * \cdots .$$
Clearly, we have $\mu = \mu_n * \mu_{>n}$.
We also define a push-forward measure of $\mu_{>n}$ by
\begin{equation}\label{ms_nugen}
  \nu_{>n} = \mu_{>n} \circ (R_n \cdots R_2 R_1)^{-1} =\delta_{R_{n+1}^{-1} B_{n+1}} * \delta_{(R_{n+2} R_{n+1})^{-1} B_{n+2}} * \cdots.
\end{equation}
The spectrality of the infinite convolution $\mu$ relies on the properties of $\{\nu_{>n}\}$.
The equi-positive condition and the integral periodic zero set were first used in \cite{Dutkay-Haussermann-Lai-2019} to study the spectrality of self-affine measures.

First, provided a technical condition,
we show that the existence of an equi-positive family (see Definition \ref{def-equipositive}) is sufficient for the spectrality of infinite convolutions in $\R^d$ (also see Theorem \ref{theorem-spectrality-2}).
Analogous results in $\R$ can be found in \cite{An-Fu-Lai-2019,Lu-Dong-Zhang-2022,LMW21}.

\begin{theorem}\label{theorem-spectrality}
  Let $\{(R_n,B_n)\}_{n=1}^\infty$ be a sequence of admissible pairs in $\R^d$.
  Suppose that the infinite convolution $\mu$ defined in \eqref{infinite-convolution} exists, and
  \[ \lim_{n \to \f} \|R_1^{-1} R_2^{-1} \cdots R_n^{-1}\| = 0. \]
  Let $\{\nu_{>n}\}$ be defined in \eqref{ms_nugen}.
  If there exists a subsequence  $\{ \nu_{>n_j} \}$ which is an equi-positive family,  then $\mu$ is a spectral measure with a spectrum in $\Z^d$.
\end{theorem}
\begin{remark}
  In Theorem \ref{theorem-spectrality}, the support of the infinite convolution $\mu$ could be noncompact. The authors and Miao \cite{Li-Miao-Wang-2022} constructed a class of singular spectral measures without compact support by showing equi-positivity in $\R$.
  Analogous examples can be constructed in $\R^d$ by using Theorem \ref{theorem-spectrality}.
\end{remark}

For $\mu \in \mcal{P}(\R^d)$, we write
\[ \mcal{Z}(\mu) = \set{\xi \in \R^d: \wh{\mu}(\xi+k) = 0 \text{ for all } k \in \Z^d} \]
for the \emph{integral periodic zero set} of Fourier transform of $\mu$.
Next we characterize the spectrality by using the integral periodic zero set (also see Theorem \ref{theorem-spectrality-emptyset-2}).
The following Theorem \ref{theorem-spectrality-emptyset} and Theorem \ref{theorem-emptyset} in $\R$ have been proved in \cite{LMW22}.

\begin{theorem}\label{theorem-spectrality-emptyset}
  Let $\{(R_n,B_n)\}_{n=1}^\infty$ be a sequence of admissible pairs in $\R^d$.
  Suppose that the infinite convolution $\mu$ defined in \eqref{infinite-convolution} exists, and
  \[ \lim_{n \to \f} \|R_1^{-1} R_2^{-1} \cdots R_n^{-1}\| = 0. \]
  Let $\{\nu_{>n}\}$ be defined in \eqref{ms_nugen}.
  If there exists a subsequence  $\{ \nu_{>n_j} \}$ which converges weakly to $\nu$ with $\mathcal{Z}(\nu) =\emptyset$, then $\mu$ is a spectral measure with a spectrum in $\Z^d$.
\end{theorem}

Generally, it is a challenging problem to compute the integral periodic zero set of Fourier transform.
We provide a sufficient condition for the integral periodic zero set of Fourier transform to be empty (also see Theorem \ref{theorem-emptyset-2}).

\begin{theorem}\label{theorem-emptyset}
Let $\mu \in \mathcal{P}(\R^d)$.
Suppose there exists a Borel subset $E \sse \R^d$ such that $\mu(E)>0$, and $$ \mu( E+k ) =0 $$ for all $k \in \Z^d\setminus\{{\bf 0}\}$.
Then we have $\mcal{Z}(\mu) = \emptyset$.
\end{theorem}
\begin{remark}
  Assume that $\mu \in \mathcal{P}(\R^d)$ has a compact support, denoted by $\mathrm{spt}(\mu)$.
  If there exists $x_0 \in \mathrm{spt}(\mu)$ such that $x_0 + k \notin \mathrm{spt}(\mu)$ for all $k \in \Z^d\setminus\{{\bf 0}\}$, then we can conclude that $\mcal{Z}(\mu) = \emptyset$.
  This is because by the compactness of $\mathrm{spt}(\mu)$, there exists a sufficiently small open neighborhood $U$ of $x_0$ such that $(U+k) \cap \mathrm{spt}(\mu) = \emptyset$ for all $k \in \Z^d\setminus\{{\bf 0}\}$.
  Then the set $E=U$ is desired in Theorem \ref{theorem-emptyset}.
\end{remark}

In the following, we apply Theorem \ref{theorem-spectrality-emptyset} and Theorem \ref{theorem-emptyset} to show the spectrality of some infinite convolutions in $\R^d$.
First, we state a general result (also see Theorem \ref{theorem-general-spectrality-2}).

\begin{theorem}\label{theorem-general-spectrality}
  Let $\{(R_n,B_n)\}_{n=1}^\infty$ be a sequence of admissible pairs in $\R^d$.
  Suppose that

  {\rm(i)} \[\lim_{n \to \f} \|R_1^{-1} R_2^{-1} \cdots R_n^{-1}\| = 0,\]

  \noindent
  and there exists a cube $C= t_0 + [0,1]^d$ for some $t_0 \in \R^d$ such that

  {\rm(ii)} for each $n \ge 1$, we have $R_n^{-1}(C + b) \subseteq C$ for every $b \in B_n$,

  {\rm(iii)} there exists an admissible pair $(R,B)$, which occurs infinitely times in the sequence $\{(R_n,B_n)\}_{n=1}^\f$, such that $R^{-1}(C + b_0) \subseteq \mathrm{int}(C)$ for some $b_0 \in B$.

  \noindent
  Then the infinite convolution $\mu$ defined in \eqref{infinite-convolution} exists and is a spectral measure with a spectrum in $\Z^d$.
\end{theorem}

Next, we focus on admissible pairs consisting of diagonal matrices and digit sets that satisfy the assumption (ii) for $C=[0,1]^d$ in Theorem \ref{theorem-general-spectrality}.
For the one-dimensional case, An, Fu, and Lai \cite{An-Fu-Lai-2019} showed that given a sequence of admissible pairs $\{(R_n,B_n)\}_{n=1}^\f$ in $\R$, if $R_n \ge 2$ and $B_n \sse \{0,1,\ldots, R_n -1\}$ for all $n \ge 1$, and \[ \liminf_{n \to \f} \# B_n < \f, \] then the infinite convolution $\mu$ defined in \eqref{infinite-convolution} is a spectral measure with a spectrum in $\Z$.
We generalize their result in the higher dimension.

In $\R^d$, we use $\mcal{D}_d$ to denote the set of all pairs $(R,B)$ of a diagonal matrix $R=\mathrm{diag}(m_1,m_2,\cdots,m_d)$, where $m_1,m_2,\ldots, m_d \ge 2$ are integers, and a nonempty subset
\[ B \sse \{ 0,1,\ldots,m_1-1 \} \times \{ 0,1,\ldots,m_2 -1 \} \times \cdots \times \{0,1,\ldots,m_d-1\}. \]
Given a sequence $\{(R_n, B_n)\}_{n=1}^\f \subseteq \mcal{D}_d$, by Theorem 1.1 in \cite{Li-Miao-Wang-2022}, the infinite convolution $\mu$ defined in \eqref{infinite-convolution} exists.
If, moreover, $\{(R_n, B_n)\}_{n=1}^\f$ is a sequence of admissible pairs, then we obtain the spectrality of $\mu$ under some conditions (also see Theorem \ref{theorem-Rd-1-2}).

\begin{theorem}\label{theorem-Rd-1}
  In $\R^d$, suppose that $\{(R_n, B_n)\}_{n=1}^\f \subseteq \mcal{D}_d$ is a sequence of admissible pairs.
  If there exists an admissible pair $(R,B)$ that occurs infinitely times in the sequence $\{(R_n,B_n)\}_{n=1}^\f$, and moreover, all diagonal elements of $R$ are greater than or equal to $d+1$,
  then the infinite convolution $\mu$ defined in \eqref{infinite-convolution} is a spectral measure with a spectrum in $\Z^d$.
\end{theorem}
\begin{remark}
  The value $d+1$ is a necessarily technical condition appearing in Theorem \ref{key-theorem}.
\end{remark}

In particular, if $\{(R_n,B_n)\}_{n=1}^\f \subseteq \mcal{D}_d$ is chosen from a finite set of admissible pairs, then the spectrality of infinite convolutions follows directly (also see Theorem \ref{theorem-Rd-2-2}).
\begin{theorem}\label{theorem-Rd-2}
  In $\R^d$, given a sequence of admissible pairs $\{(R_n, B_n)\}_{n=1}^\f \subseteq \mcal{D}_d$, if
  \[ \sup_{n \ge 1} \|R_n\| < \f, \]
  then the infinite convolution $\mu$ defined in \eqref{infinite-convolution} is a spectral measure with a spectrum in $\Z^d$.
\end{theorem}

Finally, we give some examples in $\R^2$ to illustrate our results.
\begin{example}
  Let
  \[
  R_1 = \left(
          \begin{array}{cc}
            4 & 0 \\
            4 & -4 \\
          \end{array}
        \right),
  B_1 = \bigg\{ \cv{2}{0}, \cv{3}{0}, \cv{2}{1}, \cv{3}{1} \bigg\}.
  \]
  Then the discrete measure $\delta_{R_1^{-1} B_1}$ admits a spectrum
  \[ L_1 = \bigg\{ \cv{0}{0}, \cv{2}{0}, \cv{2}{-2}, \cv{4}{-2} \bigg\}. \]
  Let
  \[
  R_2 = \left(
          \begin{array}{cc}
            3 & -3 \\
            3 & 3 \\
          \end{array}
        \right),
  B_2 = \bigg\{ \cv{0}{2}, \cv{1}{2}, \cv{0}{3} \bigg\}.
  \]
  Then the discrete measure $\delta_{R_2^{-1} B_2}$ admits a spectrum
  \[ L_2 = \bigg\{ \cv{0}{0}, \cv{3}{1}, \cv{3}{-1} \bigg\}. \]
  By calculation, we have $\|R_1^{-1}\| = (1+\sqrt{5})/8 < 1$ and $\|R_2^{-1}\| = \sqrt{2}/6 < 1$.
  For $\omega = (\omega_k)_{k=1}^\f \in \{1,2\}^{\N}$, the sequence of admissible pairs $\{ (R_{\omega_k}, B_{\omega_k}) \}_{k=1}^\f$ satisfies all assumptions for $C=[0,1]^2$ in Theorem \ref{theorem-general-spectrality}.
  Therefore, for $\omega = (\omega_k)_{k=1}^\f \in \{1,2\}^{\N}$, the infinite convolution \[\mu_\omega = \delta_{ R_{\omega_1}^{-1} B_{\omega_1} } * \delta_{ (R_{\omega_2}R_{\omega_1})^{-1} B_{\omega_2} } * \cdots * \delta_{ (R_{\omega_k} \cdots R_{\omega_2} R_{\omega_1})^{-1} B_{\omega_k} } * \cdots\] is a spectral measure with a spectrum in $\Z^2$.
\end{example}

\begin{example}
  For $n \ge 1$, if $n$ is odd, let $R_n = \mathrm{diag}(3,3)$ and $B_n = \{(0,0),(0,2),(2,0)\}$; if $n$ is even, let $R_n = \mathrm{diag}(n,n)$ and $B_n = \{(0,0),(n-1,n-1)\}$.
  Then $\{(R_n,B_n)\}_{n=1}^\f$ is a sequence of admissible pairs that satisfies neither the assumption in Theorem \ref{theorem-general-spectrality} nor the assumption in Theorem \ref{theorem-Rd-2}.
  It follows from Theorem \ref{theorem-Rd-1} that the infinite convolution
  \[ \mu = \delta_{R_1^{-1} B_1} * \delta_{(R_2 R_1)^{-1} B_2} * \cdots * \delta_{(R_n \cdots R_2 R_1)^{-1} B_n} * \cdots \]
  is a spectral measure with a spectrum in $\Z^2$.
\end{example}

The rest of paper is organized as follows.
In Section \ref{sec_pre}, we first review some definitions and some known results, and then we give the proof of Theorem~\ref{theorem-spectrality}.
In Section \ref{sec_IPZ}, we study the integral periodic zero set of Fourier transform and prove Theorem \ref{theorem-spectrality-emptyset} and Theorem \ref{theorem-emptyset}.
In Section \ref{sec_key-theorem}, we prove a key theorem for the later proof.
Finally, we prove Theorem \ref{theorem-general-spectrality}, Theorem \ref{theorem-Rd-1}, and Theorem \ref{theorem-Rd-2} in Section \ref{sec_Rd}.

\section{Spectrality of infinite convolutions}\label{sec_pre}

Let $|x|$ denote the Euclidean norm of a vector $x\in \R^d$, and let $|z|$ also denote the modulus of a complex number $z\in \C$. The operator norm of a $d\times d$ real matrix $M\in M_d(\R)$ is denoted by $\|M\|$, and $M^{\mathrm{T}}$ denotes the transpose of $M$.
The open ball centred at $x$ with radius $\gamma$ in $\R^d$ is denoted by $U(x,\gamma)$.

For a finite nonzero Borel measure $\mu\in \mcal{M}(\R^d)$, the \emph{support} of $\mu$ is  defined to be the smallest closed set with full measure, and we also say that $\mu$ is \emph{concentrated on} a Borel subset $E$ of $\R^d$ if \[ \mu\big( \R^d \sm E \big)=0. \]

For a Borel probability measure $\mu \in \mcal{P}(\R^d)$,
the \emph{Fourier transform} of $\mu$ is given by
$$ \wh{\mu}(\xi) = \int_{\R^d} e^{-2\pi i \xi \cdot x} \D \mu(x),\; \xi \in \R^d. $$
Let $\mu,\mu_1,\mu_2,\ldots \in \mcal{P}(\R^d)$. We say that $\{\mu_n\}$ \textit{converges weakly} to $\mu$ if $$\lim_{n \to \f} \int_{\R^d} f(x) \D \mu_n(x) = \int_{\R^d} f(x) \D \mu(x)$$
for all $f \in C_b(\R^d),$ where $C_b(\R^d)$ is the set of all bounded continuous functions on $\R^d$.
The following well-known theorem characterizes the weak convergence of probability measures.

\begin{theorem}\label{thm_wkcvg}
  Let $\mu,\mu_1,\mu_2,\ldots \in \mcal{P}(\R^d)$. Then $\{\mu_n\}$ converges weakly to $\mu$ if and only if $\displaystyle \lim_{n \to \f} \wh{\mu}_n(\xi)=\wh{\mu}(\xi)$ for every $\xi \in \R^d$.  Moreover, if $\{\mu_n\}$ converges weakly to $\mu$, then for every $h>0$, the sequence $\wh{\mu}_n(\xi)$ converges uniformly to $ \wh{\mu}(\xi)$ for $|\xi| \le h$.
\end{theorem}

The next theorem is often used to check whether a probability measure is spectral.
The proof for compactly supported probability measures refers to Lemma 3.3 in \cite{Jorgensen-Pedersen-1998}.
Actually, it holds for all probability measures (see \cite{LMW21} for a detailed proof).

\begin{theorem}\label{criterion}
Let $\mu \in \mcal{P}(\R^d)$. A countable subset $\Lambda \sse \R^d$ is a spectrum of $\mu$ if and only if for all $\xi \in \R^d$ we have
$$
Q(\xi) = \sum_{\lambda\in \Lambda} |\wh{\mu}(\lambda + \xi)|^2=1.
$$
\end{theorem}

The spectrality is invariant under affine transformations.
For a $d\times d$ nonsingular real matrix $M$ and a vector $b\in \R^d$, we define $T_{M,b}: \R^d \to \R^d$ by $T_{M,b}(x) = Mx + b$.

\begin{lemma}\label{lemma-affine-invariant}
  Let $\nu \in \mcal{P}(\R^d)$ be a spectral measure with a spectrum $\Lambda$.
  Then $\mu = \nu \circ T_{M,b}^{-1}$ is a spectral measure with a spectrum $(M^{\mathrm{T}})^{-1} \Lambda$ for any $d\times d$ nonsingular real matrix $M$ and any vector $b \in \R^d$.
\end{lemma}
\begin{proof}
  Note that $\wh{\mu}(\xi) = e^{-2\pi i b \cdot \xi} \wh{\nu}(M^{\mathrm{T}}\xi)$.
  Let $\Lambda' = (M^{\mathrm{T}})^{-1} \Lambda$, and then we have
  \begin{align*}
    Q(\xi) & = \sum_{\lambda' \in \Lambda'} |\wh{\mu}(\xi + \lambda')|^2 \\
     & = \sum_{\lambda' \in \Lambda'} |e^{-2\pi i b \cdot (\xi+\lambda')} \wh{\nu}\big( M^{\mathrm{T}}(\xi+\lambda') \big)|^2 \\
     & = \sum_{\lambda \in \Lambda} |\wh{\nu}( M^{\mathrm{T}}\xi+ \lambda ) |^2 \\
     & = 1.
  \end{align*}
  It follows from Theorem \ref{criterion} that $\Lambda' = (M^{\mathrm{T}})^{-1} \Lambda$ is a spectrum of $\mu$.
\end{proof}

Next lemma allows us to construct new spectral measures by convolution. One can also see Theorem 1.5 in \cite{He-Lai-Lau-2013}.

\begin{lemma}\label{lemma-spectrality-convolution}
  Suppose that $\delta_{A}$ admits a spectrum $L \sse \R^d$, where $A\sse \R^d$ is a finite subset.
  If $\nu \in \mcal{P}(\R^d)$ admits a spectrum $\Lambda \sse \R^d$ such that
  \[ a \cdot \lambda \in \Z \text{ for all $a \in A$ and $\lambda \in \Lambda$},\]
  then $\mu= \delta_{A} * \nu$ is a spectral measure with a spectrum $L + \Lambda$.
\end{lemma}
\begin{proof}
  Note that \[ \wh{\delta}_A(\xi) = \frac{1}{\# A} \sum_{a\in A} e^{-2\pi i a\cdot \xi}. \]
  Since $a \cdot \lambda \in \Z$ for all $a \in A$ and $\lambda \in \Lambda$, we have
  \[ \wh{\delta}_A(\xi+\lambda) = \wh{\delta}_A(\xi) \text{ for all $\lambda \in \Lambda$ and $\xi \in \R^d$}. \]
  Let $\Lambda' = L + \Lambda$, and then by Theorem \ref{criterion} we have
  \begin{align*}
    Q(\xi) & = \sum_{\lambda' \in \Lambda'} |\wh{\mu}(\xi + \lambda')|^2 = \sum_{\lambda' \in \Lambda'} |\wh{\delta}_A(\xi + \lambda')\wh{\nu}(\xi + \lambda')|^2 \\
     & = \sum_{\ell \in L} \sum_{\lambda \in \Lambda} |\wh{\delta}_A(\xi + \ell+ \lambda)\wh{\nu}(\xi + \ell+\lambda)|^2 \\
     & = \sum_{\ell \in L} \sum_{\lambda \in \Lambda} |\wh{\delta}_A(\xi + \ell)\wh{\nu}(\xi + \ell+\lambda)|^2 \\
     & = \sum_{\ell \in L} |\wh{\delta}_A(\xi + \ell)|^2 \\
     & = 1.
  \end{align*}
  It follows from Theorem \ref{criterion} that $\Lambda' = L + \Lambda$ is a spectrum of $\mu$.
\end{proof}

\begin{corollary}
  If $\delta_{A}$ is a spectral measure, where $A\sse \Z^d$ is a finite subset, and $\nu \in \mcal{P}(\R^d)$ be a spectral measure with a spectrum in $\Z^d$, then $\mu= \delta_{A} * \nu$ is a spectral measure.
\end{corollary}

We list some useful properties of admissible pairs in the following lemma, see \cite{Dutkay-Haussermann-Lai-2019, Laba-Wang-2002} for details.

\begin{lemma}\label{lemma-HT}
  Suppose that $(R,B)$ is an admissible pair in $\R^d$ and the discrete measure $\delta_{R^{-1} B}$ admits a spectrum $L \sse \Z^d$. Then

  {\rm(i)} $L + \ell_0$ is also a spectrum of $\delta_{R^{-1}B}$ for all $\ell_0 \in \R^d$.

  {\rm(ii)} The elements in $L$ are distinct modulo $R^{\mathrm{T}} \Z^d$, and if $\wt{L} \equiv L \pmod{ R^\mathrm{T}\Z^d}$, then $\wt{L}$ is also a spectrum of $\delta_{R^{-1}B}$.

  {\rm(iii)} For $1 \le j \le n$, let $(R_j, B_j)$ be an admissible pairs in $\R^d$ and let $L_j \sse \Z^d$ be a spectrum of $\delta_{R_j^{-1} B_j}$.
  Write $$\mathbf{R}= R_n R_{n-1} \cdots R_1,\quad \mathbf{B} = (R_n R_{n-1} \cdots R_2) B_1 + \cdots + R_n B_{n-1} + B_n.$$
  Then $(\mathbf{R},\mathbf{B})$ is an admissible pair, and $\delta_{\mathbf{R}^{-1}\mathbf{B}}$ admits a spectrum
  $$\mathbf{L} = L_1 + R_1^{\mathrm{T}} L_2 + \cdots + (R_1^{\mathrm{T}} R_2^{\mathrm{T}} \cdots R_{n-1}^{\mathrm{T}}) L_n.$$
\end{lemma}

Next, we state the definition of equi-positivity, and then prove Theorem~\ref{theorem-spectrality}.

\begin{definition}\label{def-equipositive}
  We call $\Phi \sse \mcal{P}(\R^d)$ an equi-positive family if there exists $\ep>0$ and $\gamma>0$ such that for  $x\in [0,1)^d$ and $\mu\in \Phi$, there is an integral vector $k_{x,\mu} \in \Z^d$ such that
  $$ |\wh{\mu}(x+y+k_{x,\mu})| \ge \ep$$
  for all $|y| <\gamma,$ where $k_{x,\mu} ={\bf 0}$ for $x={\bf 0}$.
\end{definition}

\begin{theorem}\label{theorem-spectrality-2}
  Let $\{(R_n,B_n)\}_{n=1}^\infty$ be a sequence of admissible pairs in $\R^d$.
  Suppose that the infinite convolution $\mu$ defined in \eqref{infinite-convolution} exists, and
  \[ \lim_{n \to \f} \|R_1^{-1} R_2^{-1} \cdots R_n^{-1}\| = 0. \]
  Let $\{\nu_{>n}\}$ be defined in \eqref{ms_nugen}.
  If there exists a subsequence $\{ \nu_{>n_j} \}$ which is an equi-positive family,  then $\mu$ is a spectral measure with a spectrum in $\Z^d$.
\end{theorem}
\begin{proof}
  Since $\{(R_n,B_n)\}_{n=1}^\infty$ is a sequence of admissible pairs in $\R^d$, by definition, for each $n \ge 1$, the discrete measure $\delta_{R_n^{-1} B_n}$ admits a spectrum $L_n \sse \Z^d$. By Lemma \ref{lemma-HT} $(\mathrm{i})$,  we may assume that ${\bf 0}\in L_n$ for all $n\ge 1$.

  Since the family $\{\nu_{>n_j}\}$ is equi-positive, there exists $\ep>0$ and $\gamma>0$ such that for $x\in [0,1)^d$ and $j \ge 1$, there is an integral vector $k_{x,j} \in \Z^d$ such that
  $$ |\wh{\nu}_{>n_j}(x+y+k_{x,j})| \ge \ep$$
  for all $|y| <\gamma,$
  and $k_{x,j}={\bf 0}$ for $x={\bf 0}$.

  For $q > p \ge 0$, we define $\mathbf{R}_{p,q} = R_q R_{q-1}\cdots R_{p+1} $,
  $$ \mathbf{B}_{p,q}= R_q R_{q-1}\cdots R_{p+2} B_{p+1} + R_q R_{q-1}\cdots R_{p+3} B_{p+2} + \cdots + R_q B_{q-1} + B_q, $$
  and $$\mathbf{L}_{p,q} = L_{p+1} + R_{p+1}^{\mathrm{T}} L_{p+2} + \cdots + ( R_{p+1}^{\mathrm{T}} R_{p+2}^{\mathrm{T}} \cdots R_{q-1}^{\mathrm{T}}) L_q.$$
 By Lemma \ref{lemma-HT} $(\mathrm{iii})$, $\mathbf{L}_{p,q}$ is a spectrum of $\delta_{ \mathbf{R}_{p,q}^{-1} \mathbf{B}_{p,q} }$.

  We construct a sequence of finite subsets $\Lambda_j \sse \Z^d$ for $j \ge 1$ by induction.
  Let $m_1 = n_1$ and $\Lambda_1 = \mathbf{L}_{0, m_1}$. Note that ${\bf 0}\in \Lambda_1$ and $\Lambda_1$ is a spectrum of $\mu_{m_1}$.
  For $j \ge 2$, suppose that $\Lambda_{j-1}$ has been defined with ${\bf 0} \in \Lambda_{j-1}$ and $\Lambda_{j-1}$ is a spectrum of $\mu_{m_{j-1}}$.
  Since $\| (R_1^T R_2^T \cdots R_n^T)^{-1}\| = \| R_1^{-1} R_2^{-1} \cdots R_n^{-1}\| $ tends to $0$ as $n \to \f$,
  we may choose a sufficiently  large  integer $m_j$ in the sequence $\set{n_j}$ such that $m_j > m_{j-1}$ and for all $ \lambda \in \Lambda_{j-1}$,
  \begin{equation}\label{ineqla}
   \Big| \big(R_1^\mathrm{T} R_2^\mathrm{T} \cdots R_{m_j}^\mathrm{T}\big)^{-1} \lambda \Big| < \frac{\gamma}{2}.
  \end{equation}
  Now we define
  \begin{equation} \label{defLam}
    \Lambda_j = \Lambda_{j-1} + \mathbf{R}_{0,m_{j-1}}^{\mathrm{T}} \set{ \lambda + \mathbf{R}_{m_{j-1}, m_j}^{\mathrm{T}} k_{\lambda,j}  : \lambda \in \mathbf{L}_{m_{j-1}, m_j} },
  \end{equation}
  where, by the equi-positivity of $\{ \nu_{>n_j} \}$, the integral vectors $k_{\lambda,j}\in \Z^d$ are chosen to satisfy
  \begin{equation}\label{lowerbound}
    \left| \wh{\nu}_{>m_j}\left( (R_{m_{j-1}+1}^\mathrm{T} \cdots R_{m_j}^\mathrm{T})^{-1} \lambda + y + k_{\lambda,j} \right) \right|\ge \ep
  \end{equation}
  for all $|y|<\gamma$,  and $k_{\lambda,j} ={\bf 0}$ for $\lambda={\bf 0}$.
  Since $\Lambda_{j-1}$ is a spectrum of $\mu_{m_{j-1}} = \delta_{\mathbf{R}_{0,m_{j-1}}^{-1} \mathbf{B}_{0,m_{j-1}}}$, by Lemma \ref{lemma-HT} (ii) and (iii), we conclude that $\Lambda_j$ is a spectrum of
  $$\mu_{m_j} = \delta_{\mathbf{R}_{0,m_{j-1}}^{-1} \mathbf{B}_{0,m_{j-1}}} * \delta_{(\mathbf{R}_{m_{j-1}, m_j} \mathbf{R}_{0,m_{j-1}} )^{-1} \mathbf{B}_{m_{j-1},m_j}}. $$
  Since ${\bf 0}\in \mathbf{L}_{m_{j-1}, m_j}$ and ${\bf 0}\in \Lambda_{j-1}$, we have ${\bf 0}\in \Lambda_j$ and $\Lambda_{j-1} \sse \Lambda_j$.

  We write $$\Lambda = \bigcup_{j=1}^\f \Lambda_j,$$ and prove that $\Lambda $ is a spectrum of $\mu$.
  By Theorem~\ref{criterion}, it is equivalent to show that for all $ \xi \in \R^d$, $$Q(\xi) = \sum_{\lambda \in \Lambda} |\wh{\mu}(\lambda + \xi)|^2=1.$$

  For $\xi\in \R^d$, since $\Lambda_j$ is a spectrum of $\mu_{m_j}$, by Theorem~\ref{criterion}, we have
  \begin{equation}\label{mu-m-j}
    \sum_{\lambda \in \Lambda_j} \left| \wh{\mu}_{m_j}(\lambda + \xi) \right|^2 =1.
  \end{equation}
  It follows that
  \begin{align*}
    \sum_{\lambda \in \Lambda_j} \left| \wh{\mu}(\lambda + \xi) \right|^2
    &= \sum_{\lambda \in \Lambda_j} \left|\wh{\mu}_{m_j}(\lambda + \xi) \right|^2 \left|\wh{\mu}_{>m_j}(\lambda + \xi)\right|^2 \\
    &\le \sum_{\lambda \in \Lambda_j} \left| \wh{\mu}_{m_j}(\lambda + \xi) \right|^2 \\
    &\le 1.
  \end{align*}
  Letting $j$ tend to the infinity, we obtain that
  \begin{equation}\label{ineqQ}
    Q(\xi) \le 1
  \end{equation}
  for all  $\xi \in \R^d$.

  Fix $\xi\in \R^d$.  We define
  $$f(\lambda) = |\wh{\mu}(\lambda+ \xi)|^2, \;\lambda \in \Lambda,$$
  and
  $$ f_j(\lambda) =
     \begin{cases}
       |\wh{\mu}_{m_j}(\xi+\lambda)|^2, & \mbox{if } \lambda \in \Lambda_j, \\
       0, & \mbox{if } \lambda \in \Lambda \sm \Lambda_j,
     \end{cases}
  $$
  for each $j \ge 1$.
  For each $\lambda \in \Lambda$, there exists  $j_\lambda \ge 1$ such that $\lambda \in \Lambda_j$ for $j \ge j_\lambda$, and hence $$ \lim_{j \to \f} f_{j}(\lambda) =  \lim_{j \to \f} |\wh{\mu}_{m_j}(\xi+\lambda)|^2 =f(\lambda), $$
  where the last equality follows from Theorem \ref{thm_wkcvg} and the fact that $\{\mu_{m_j}\}$ converges weakly to $\mu$.
  Choose an integer $j_0 \ge 1$ sufficiently large such that for $j > j_0$
  \begin{equation}\label{ineqxi}
    \left| (R_1^\mathrm{T} R_2^\mathrm{T} \cdots R_{m_{j}}^\mathrm{T})^{-1} \xi \right| < \frac{\gamma}{2}.
  \end{equation}
  For each $\lambda \in \Lambda_j$ where $j > j_0$, by~\eqref{defLam}, we have that
  $$\lambda= \lambda_1 + (R_{m_{j-1}} \cdots R_2 R_1)^{\mathrm{T}} \lambda_2 + (R_{m_{j}} \cdots R_2 R_1)^{\mathrm{T}} k_{\lambda_2,j},$$
  where $\lambda_1 \in \Lambda_{j-1}$ and $\lambda_2\in \mathbf{L}_{m_{j-1}, m_j}$.
  By ~\eqref{ineqla} and \eqref{ineqxi}, we have that
  $$ \left| (R_1^\mathrm{T} R_2^\mathrm{T} \cdots R_{m_j}^\mathrm{T})^{-1}(\lambda_1+\xi) \right| < \gamma. $$
  It follows from \eqref{lowerbound} that
  \begin{align*}
    f(\lambda) & = |\wh{\mu}(\lambda + \xi)|^2 =  \left| \wh{\mu}_{m_j}(\lambda + \xi)\right|^2 \left| \wh{\mu}_{>m_j}(\lambda + \xi)\right|^2 \\
    &= \left| \wh{\mu}_{m_j}(\lambda + \xi)\right|^2 \left| \wh{\nu}_{>m_j}\left( (R_1^{\mathrm{T}} R_2^{\mathrm{T}} \cdots R_{m_j}^{\mathrm{T}})^{-1}(\lambda +\xi) \right) \right|^2 \\
    & = \left|\wh{\mu}_{m_j}(\lambda + \xi)\right|^2 \left| \wh{\nu}_{>m_j}\left( (R_{m_{j-1}+1}^\mathrm{T} \cdots R_{m_j}^\mathrm{T})^{-1} \lambda_2 + (R_1^\mathrm{T} \cdots R_{m_j}^\mathrm{T})^{-1}(\lambda_1+\xi) + k_{\lambda_2,j}\right) \right|^2 \\
    & \ge \ep^2 f_j(\lambda).
  \end{align*}
  Therefore, for $j > j_0$, $$f_j(\lambda) \le \ep^{-2} f(\lambda)$$
  for all $\lambda \in \Lambda$.

  Let $\rho$ be the counting measure on the set $\Lambda$.  We have that
  $$ \int_\Lambda f(\lambda) \D \rho(\lambda) = \sum_{\lambda \in \Lambda} |\wh{\mu}(\lambda + \xi)|^2 = Q(\xi) .
  $$
  By ~\eqref{ineqQ}, $f(\lambda)$ is integrable with respect to the counting measure $\rho$.
  Applying the dominated convergence theorem and \eqref{mu-m-j}, we obtain that
  \begin{align*}
    Q(\xi) &= \lim_{j \to \f} \int_\Lambda f_j(\lambda) \D \rho(\lambda)  \\
    &= \lim_{j \to \f} \sum_{\lambda \in \Lambda_j} |\wh{\mu}_{m_j}(\lambda + \xi)|^2   \\
    &=1.
  \end{align*}
  Hence, by Theorem \ref{criterion}, $\Lambda$ is a spectrum of $\mu$, and $\mu$ is a spectral measure.
\end{proof}

\section{Integral periodic zero set}\label{sec_IPZ}

First, we give the proof of Theorem~\ref{theorem-spectrality-emptyset}.
\begin{theorem}\label{theorem-spectrality-emptyset-2}
  Let $\{(R_n,B_n)\}_{n=1}^\infty$ be a sequence of admissible pairs in $\R^d$.
  Suppose that the infinite convolution $\mu$ defined in \eqref{infinite-convolution} exists, and
  \[ \lim_{n \to \f} \|R_1^{-1} R_2^{-1} \cdots R_n^{-1}\| = 0. \]
  Let $\{\nu_{>n}\}$ be defined in \eqref{ms_nugen}.
  If there exists a subsequence  $\{ \nu_{>n_j} \}$ which converges weakly to $\nu$ with $\mathcal{Z}(\nu) =\emptyset$, then $\mu$ is a spectral measure with a spectrum in $\Z^d$.
\end{theorem}
\begin{proof}
  By Theorem \ref{theorem-spectrality}, it suffices to show that the family $\{\nu_{> n_j}\}_{j=j_0}^\f$ is equi-positive for some large $j_0 \ge 1$.

  Since $\mathcal{Z}(\nu)=\emptyset$, for each $x\in [0,1]^d$, there exists $k_x\in \Z^d$ such that $\wh{\nu}(x+k_x) \ne 0$.
  Thus, there exists $\ep_x>0$ and $\gamma_x >0$ such that
  \begin{equation}\label{positive-1}
    |\wh{\nu}(x+k_x+y)|\ge \ep_x
  \end{equation}
  for all $|y|<\gamma_x$.
  Note that
  $$[0,1]^d \sse \bigcup_{x\in[0,1]^d} U(x,\gamma_x/2).$$
  By the compactness of $[0,1]^d$, there exist finitely many $x_1, x_2, \ldots, x_q \in [0,1]^d$ such that
  \begin{equation}\label{finite-cover}
    [0,1]^d \sse \bigcup_{\ell=1}^q U(x_\ell, \gamma_{x_\ell}/2).
  \end{equation}
  Since $\wh{\nu}(0) =1$ and $\wh{\nu}(\xi)$ is continuous, there exists $\gamma_0>0$ such that
  \begin{equation}\label{positive-2}
    |\wh{\nu}(y)| \ge 1/2
  \end{equation}
  for all $|y| <\gamma_0$.

  Let $\ep = \min\set{ 1/4, \ep_{x_1}/2, \ep_{x_2}/2, \ldots, \ep_{x_q}/2 }$ and $\gamma = \min \set{ \gamma_0, \gamma_{x_1}/2, \gamma_{x_2}/2, \ldots, \gamma_{x_q}/2 }$.
  Let $h = \sqrt{d} + \gamma+\max\set{|k_{x_1}|, |k_{x_2}|, \ldots, |k_{x_q}|}$.
  Since $\{ \nu_{>n_j} \}$ converges weakly to $\nu$, by Theorem \ref{thm_wkcvg}, we have $\wh{\nu}_{>n_j}(\xi)$ converges uniformly to $\wh{\nu}(\xi)$ for $|\xi| \le h$.
  Thus, there exists $j_0 \ge 1$ such that
  \begin{equation}\label{converge-uniformly}
    |\wh{\nu}_{>n_j}(\xi) - \wh{\nu}(\xi)| < \ep
  \end{equation}
  for all $j \ge j_0$ and all $|\xi| \le h$.

  For each $x\in [0,1)^d \setminus\{{\bf 0}\}$, by \eqref{finite-cover}, we may find $1\le \ell \le q$ such that $|x-x_\ell|<\gamma_{x_\ell}/2$.
  For $j \ge j_0$ and $|y|<\gamma$, noting that $|x+k_{x_\ell} +y|<h$, it follows from \eqref{converge-uniformly} that $$|\wh{\nu}_{>n_j}(x+k_{x_\ell} + y)| \ge |\wh{\nu}(x+k_{x_\ell} +y)| -\ep. $$
  Since $|x-x_\ell +y| < \gamma_{x_\ell}/2 +\gamma \le \gamma_{x_\ell} $, by \eqref{positive-1}, we have that $$|\wh{\nu}(x+k_{x_\ell} +y)|= |\wh{\nu}(x_\ell+k_{x_\ell} +x-x_\ell+y)|\ge \ep_{x_\ell} \ge 2\ep.$$
  Thus, for $j \ge j_0$ and $|y|<\gamma$, $$|\wh{\nu}_{>n_j}(x+k_{x_\ell} + y)| \ge \ep. $$
  For $x={\bf 0}$, it follows from \eqref{positive-2} and \eqref{converge-uniformly}  that for $j \ge j_0$ and for $|y| <\gamma$, $$|\wh{\nu}_{>n_j}(y)| \ge |\wh{\nu}(y)|-\ep \ge 1/4 \ge \ep.$$
  Therefore, the family $\{ \nu_{> n_j} \}_{j=j_0}^\f$ is equi-positive.
\end{proof}

Next, we study the integral periodic zero set of Fourier transform. Recall that the integral periodic zero set is given by
\[ \mcal{Z}(\mu) = \set{\xi \in \R^d: \wh{\mu}(\xi+k) = 0 \text{ for all } k \in \Z^d}. \]

Let $\mathbb{T}^d = \R^d / \Z^d$, and we write $\mathcal{M}(\mathbb{T}^d)$ for the set of all complex Borel measures on $\mathbb{T}^d$.
The following theorem is the uniqueness theorem of Fourier coefficients which is often contained in Fourier analysis or harmonic analysis.
Here we give a proof for the readers' convenience.

\begin{theorem}\label{thm_fcof}
  Let $\nu \in \mathcal{M}(\mathbb{T}^d)$.
  If the Fourier coefficient
  $$
  \widehat{\nu}(k) = \int_{\mathbb{T}^d} e^{-2\pi i k \cdot x} \D \nu(x) =0
  $$
  for all $k \in \Z^d$, then $\nu=0$.
\end{theorem}
\begin{proof}
   By the Stone-Weierstrass theorem, every continuous function on $\mathbb{T}^d$ can be approximated uniformly by trigonometric polynomials.
   Since all Fourier coefficients $\wh{\nu}(k) =0$ for $k \in \Z^d$, it follows that
   $$ \int_{\mathbb{T}^d} f(x) \D \nu(x) =0 $$
   for every continuous function $f$ on $\mathbb{T}^d$.
   By the Riesz representation theorem \cite[Theorem 6.19]{Rudin-1987}, we conclude that $\nu =0$.
\end{proof}

We apply the uniqueness theorem of Fourier coefficients to prove Theorem \ref{theorem-emptyset}.
\begin{theorem}\label{theorem-emptyset-2}
Let $\mu \in \mathcal{P}(\R^d)$.
Suppose there exists a Borel subset $E \sse \R^d$ such that $\mu(E)>0$, and $$ \mu( E+k ) =0 $$ for all $k \in \Z^d\setminus\{{\bf 0}\}$.
 Then we have $\mcal{Z}(\mu) = \emptyset$.
\end{theorem}
\begin{proof}
  For $k=(k_1, \ldots, k_d) \in \Z^d$, we write
  $$
  C_k=[k_1,k_1+1)\times\cdots \times [k_d, k_d+1).
  $$
  Since $\mu(E) >0$, there exists $k_0 \in \Z^d$ such that $\mu( E \cap C_{k_0})>0$.
  Replacing the set $E$ by $E\cap C_{k_0}$, we may assume that $E \sse C_{k_0}$ for some $k_0 \in \Z^d$.
  Let $\widetilde{E}= E - k_0$ and $\widetilde{\mu} = \mu*\delta_{\{-k_0\}}$. Then $\widetilde{E} \sse [0,1)^d$. Note that for all Borel subset $F \sse \R^d$, we have that
  $$
  \widetilde{\mu}(F) = \mu*\delta_{\{-k_0\}}(F) = \mu(F+k_0).
  $$
  It follows that $\widetilde{\mu}\big( \widetilde{E} \big) = \mu(E)>0$, and $\widetilde{\mu}\big( \widetilde{E} +k \big) = \mu(E +k)=0$ for all $k \in \Z^d \setminus \{0\}$. Since $\widehat{\widetilde{\mu}}(\xi) = e^{2 \pi i k_0 \cdot \xi} \widehat{\mu}(\xi)$, we have $\mathcal{Z} ( \widetilde{\mu} ) = \mathcal{Z}(\mu)$.
  Therefore, we assume that $E \sse [0,1)^d$ in the following.

  For $\xi \in \R^d$, we define a complex measure $\mu_\xi$ on $\R^d$ by
\begin{equation}\label{muxi}
\D \mu_\xi = e^{-2\pi i\xi \cdot x} \D \mu.
\end{equation}
Consider the natural homomorphism $\pi: \R^d \to \mathbb{T}^d$, and let $\rho_\xi = \mu_\xi \circ \pi^{-1}$ be the image measure on $\mathbb{T}^d$ of $\mu_\xi$ by $\pi$, i.e., for each Borel subset $F \subseteq \mathbb{T}^d$,
  \begin{equation}\label{rhoxi}
  \rho_\xi(F) = \mu_\xi(F+\Z^d) = \sum_{k\in \Z^d}\mu_\xi(F+k).
  \end{equation}
Since $\mu(E)>0$, we write
  $$
  \nu(\;\cdot\;) = \frac{1}{\mu(E)} \mu(\;\cdot\; \cap E)
  $$
  for the normalized measure of $\mu$ on $E$.

Suppose that $\mathcal{Z}(\mu) \ne \emptyset$.  Arbitrarily choose $\xi_0 \in \mathcal{Z}(\mu)$, and we have that $\widehat{\mu}(\xi_0+k) =0$ for all $k \in \Z^d$. This implies that
\begin{align*}
 \widehat{\rho}_{\xi_0}(k) & = \int_{\mathbb{T}^d} e^{-2\pi i k \cdot x} \D \mu_{\xi_0}\circ \pi^{-1}(x)  \\
    & = \int_{\R^d} e^{-2\pi i k \cdot \pi(x)} \D \mu_{\xi_0}(x) \\
    & = \int_{\R^d} e^{-2\pi i k \cdot x} e^{-2\pi i \xi_0 \cdot x} \D \mu(x) \\
    & = \widehat{\mu}(\xi_0+k)\\
    & = 0
\end{align*}
for all $k \in \Z^d$. By Theorem \ref{thm_fcof}, we conclude that $\rho_{\xi_0} =0$.
By~\eqref{muxi} and \eqref{rhoxi}, it follows that
$$\int_{E + \Z^d} e^{-2\pi i \xi_0 \cdot x} \D\mu(x)=\mu_{\xi_0}(E+\Z^d) = \rho_{\xi_0}(E) =  0.$$
Since $\mu(E + k) =0$ for all $k \in \Z^d \setminus\{{\bf 0}\}$, we obtain that $$ \int_E e^{-2\pi i \xi_0 \cdot x} \D\mu(x) =0.$$
It follows that $$\widehat{\nu}(\xi_0) = \frac{1}{\mu(E)} \int_E e^{-2\pi i \xi_0 \cdot x} \D\mu(x) =0. $$
Note that $\xi_0+k \in \mathcal{Z}(\mu)$ for all $k\in \Z^d$.
Thus, we have that for all $k \in \Z^d$,
$$
\widehat{\nu}(\xi_0+k) =0.
$$

Let $\nu_{\xi_0}$ be defined by
  \begin{equation}\label{nu-xi}
    \D \nu_{\xi_0} = e^{-2\pi i \xi_0 \cdot x} \D \nu.
  \end{equation}
  Since $\nu$ is concentrated on $E \sse [0,1)^d$, $\nu_{\xi_0}$ can be viewed as a complex measure on $\mathbb{T}^d$.  Moreover, the Fourier coefficients $\widehat{\nu}_{\xi_0}(k) = \widehat{\nu}(\xi_0+k)=0$ for all $k \in \Z^d$.
  By Theorem \ref{thm_fcof}, we have that $\nu_{\xi_0}=0$.
 But, by~\eqref{nu-xi} and Theorem 6.13 in \cite{Rudin-1987}, we have the total variation $|\nu_{\xi_0}| = \nu \ne 0$,
and this leads to a contradiction.

Therefore, we conclude that $\mathcal{Z}(\mu) =\emptyset$.
\end{proof}

The following conclusion is an immediate consequence. 

\begin{corollary}
  Let $\mu \in \mcal{P}(\R^d)$ with $\mathrm{spt}(\mu) \sse [0,1]^d$. If $\mcal{Z}(\mu) \ne \emptyset$, then we have \[ \mathrm{spt}(\mu) \sse [0,1]^d \setminus (0,1)^d. \]
\end{corollary}

\begin{remark}
  It has been showed in \cite{An-Fu-Lai-2019} that for $\mu \in \mathcal{P}(\R)$ with $\mathrm{spt}(\mu) \sse [0,1]$, $\mathcal{Z}(\mu) \ne \emptyset$ if and only if $\mu = \frac{1}{2} \delta_0 + \frac{1}{2} \delta_1$.
\end{remark}

\section{A key ingredient}\label{sec_key-theorem}

In order to prove Theorem \ref{theorem-Rd-1} and Theorem \ref{theorem-Rd-2}, we need the following Theorem \ref{key-theorem} to show that the assumption in Theorem \ref{theorem-emptyset} is satisfied by a class of measures.
Recall that $\mcal{M}(\R^d)$ denotes the collection of all finite nonzero Borel measures on $\R^d$.

\begin{theorem}\label{key-theorem}
  Let $(R,B) \in \mcal{D}_d$ where $R = \mathrm{diag}(m_1, m_2, \ldots, m_d)$ with $m_1,m_2,\ldots,m_d \ge d+1$.
Associated with the digit set $B$, a positive weight vector $(p_b)_{b \in B}$ is given.
Let $\mu \in \mcal{M}(\R^d)$ with $\mathrm{spt}(\mu) \sse [0,1]^d$.
Set \[\nu = (\lambda* \mu) \circ R\;\;\text{where}\;\; \lambda = \sum_{b \in B} p_b \delta_{b}. \]
Then there exists a Borel subset $E\sse \R^d$ such that $\nu(E)>0$ and \[ \nu(E+n) =0  \] for all $n \in \Z^d \setminus \{{\bf 0}\}$.
\end{theorem}
\begin{remark}
  The condition that all diagonal elements of $R$ are $\ge d+1$ is necessary. Here we give an example in $\R^2$.
  Let \[
  R = \left(
          \begin{array}{cc}
            2 & 0 \\
            0 & 2 \\
          \end{array}
        \right),~
  B = \Bigg\{ \cv{1}{0}, \cv{0}{1} \Bigg\},~
  \mu = \frac{1}{2} \delta_{\{(0,0)\}} + \frac{1}{2} \delta_{\{(1,1)\}}.
  \]
  Set $\nu=(\delta_B *\mu)\circ R$. By simple calculation, we have
  \[ \nu = \frac{1}{4} \delta_{\{(0,1/2)\}} + \frac{1}{4} \delta_{\{(1,1/2)\}} + \frac{1}{4} \delta_{\{(1/2,0)\}} + \frac{1}{4} \delta_{\{(1/2,1)\}}.\]
  The conclusion in Theorem \ref{key-theorem} fails, and moreover, we have $(1/2,1/2) \in \mathcal{Z}(\nu)$.
\end{remark}

For later application, we note that
\begin{equation}\label{support-nu}
  \mathrm{spt}(\nu) = R^{-1} B + R^{-1} \mathrm{spt}(\mu) \sse [0,1]^d.
\end{equation}
We first show Theorem \ref{key-theorem} in $\R$.

\begin{lemma}\label{lemma-dimension-1}
  Theorem \ref{key-theorem} holds in $\R$.
\end{lemma}
\begin{proof}
  For $d = 1$, note that $\mathrm{spt}(\nu)\sse [0,1]$, and $R \ge 2$ is an integer.
  If $\mathrm{spt}(\nu) \cap (0,1) \ne \emptyset$, then the set $E = \mathrm{spt}(\nu) \cap (0,1)$ is desired.
  Otherwise, we have that $\mathrm{spt}(\nu) \sse \{0,1\}$.

  If both $0$ and $1$ are in $\mathrm{spt}(\nu)$, then by \eqref{support-nu} there exist $b_1,b_2 \in B$ such that
  \[ 0 \in R^{-1} b_1 + R^{-1} \mathrm{spt}(\mu) \text{ and } 1 \in R^{-1} b_2 + R^{-1} \mathrm{spt}(\mu). \]
  Since $\mathrm{spt}(\mu)\sse [0,1]$, it follows that \[ b_1 =0, \; b_2 = R-1,\; \{0,1\} \sse \mathrm{spt}(\mu). \]
  By \eqref{support-nu}, we obtain $1/R \in \mathrm{spt}(\nu). $
  Since $R \ge 2$, we have $0< 1/R < 1$. This yields a contradiction.
  Thus, we conclude that $\mathrm{spt}(\nu) = \{0\}$ or $\mathrm{spt}(\nu)= \{1\}$.
  Then the set $E =\mathrm{spt}(\nu)$ is desired.
\end{proof}

Next, we prove Theorem \ref{key-theorem} by induction on the dimension $d$.
\begin{proof}[Proof of Theorem \ref{key-theorem}]
  We assume that Theorem \ref{key-theorem} holds in $\R^{d-1}$ for some $d\ge 2$.
  In the following, we prove that Theorem \ref{key-theorem} also holds in $\R^{d}$.
  The proof is split into two cases.

  \textbf{Case I}: there exists $1\le j \le d$ such that $\mu\big( \{ (x_1,\ldots,x_d) \in [0,1]^d: 0< x_j < 1 \} \big) >0$.
  Without loss of generality, we assume that $\mu\big([0,1]^{d-1} \times (0,1)\big)>0$.

  Let $\mu_1$ and $\mu_2$ be the restrictions of $\mu$ on $[0,1]^{d-1} \times (0,1)$ and $[0,1]^{d-1} \times \{0,1\}$, respectively.
  Then $\mu = \mu_1 + \mu_2$.
  Set $\nu_1 = (\lambda*\mu_1) \circ R$ and $\nu_2 = (\lambda*\mu_2) \circ R$.
  Then $\nu = \nu_1 + \nu_2$.
  Note that $\nu_1$ is concentrated on $[0,1]^{d-1} \times \big( [0,1] \sm \{0, 1/m_d, 2/m_d, \ldots, 1 \} \big)$, $\nu_2$ is concentrated on $[0,1]^{d-1} \times \{0, 1/m_d, 2/m_d, \ldots, 1 \}$, and these two sets are disjoint.
  Thus, $\nu_1$ and $\nu_2$ are the restrictions of $\nu$ on $[0,1]^{d-1} \times \big( [0,1] \sm \{0, 1/m_d, 2/m_d, \ldots, 1 \} \big)$ and $[0,1]^{d-1} \times \{0, 1/m_d, 2/m_d, \ldots, 1 \}$, respectively.

  Define $\pi_{d,d-1}: \R^d \to \R^{d-1}$ by $(x_1,x_2,\ldots, x_d) \mapsto (x_1,x_2,\ldots, x_{d-1})$.
  Let $\mu' = \mu_1 \circ \pi_{d,d-1}^{-1}$, $R'=\mathrm{diag}(m_1, m_2, \ldots, m_{d-1})$ and $B' = \pi_{d,d-1}(B)$.
  For $b' \in B'$, let \[ p_{b'} = \sum_{b \in \pi_{d,d-1}^{-1}(b')} p_b. \]
  Then $(p_{b'})_{b' \in B'}$ is a positive weight vector associated with $B'$.
  Let \[ \lambda' = \sum_{b' \in B'} p_{b'} \delta_{b'},\;\nu' = (\lambda'*\mu') \circ R'. \]
  Since Theorem \ref{key-theorem} holds in $\R^{d-1}$, there exists a Borel subset $E'\sse\R^{d-1}$ such that $\nu'(E')>0$ and \[ \nu'(E'+n') =0  \] for all $n' \in \Z^{d-1} \setminus \{{\bf 0}\}$.

  It's straightforward to verify that \[ \nu' = \nu_1 \circ \pi_{d,d-1}^{-1} \]
  Construct a Borel subset of $\R^d$ by \[ E=E' \times \big( [0,1] \sm \{0, 1/m_d, 2/m_d, \ldots, 1 \} \big) .\]
  Note that $\nu_1$ is concentrated on $[0,1]^{d-1} \times \big( [0,1] \sm \{0, 1/m_d, 2/m_d, \ldots, 1 \} \big)$.
  Then
  \begin{align*}
    \nu'(E') & = \nu_1\big( \pi_{d,d-1}^{-1} (E') \big) = \nu_1(E' \times \R) \\
    & = \nu_1\Big( E' \times  \big( [0,1] \sm \{0, 1/m_d, 2/m_d, \ldots, 1 \} \big) \Big) \\
    & = \nu_1(E).
  \end{align*}
  Thus, $\nu_1(E) = \nu'(E') >0$, and it follows that $\nu(E) >0$.

  Take $n=(n_1,\ldots,n_d)\in \Z^{d} \setminus \{{\bf 0}\}$.
  If $n_d\ne 0$, then $(E+n) \cap [0,1]^{d} =\emptyset$, and hence $\nu(E+n) =0$.
  If $n_d=0$, then $n'=(n_1,\ldots,n_{d-1})\in \Z^{d-1} \setminus \{{\bf 0}\}$, and we have
  \begin{align*}
    \nu_1(E+n) & = \nu_1\Big( (E'+n') \times \big( [0,1] \sm \{0, 1/m_d, 2/m_d, \ldots, 1 \} \big) \Big) \\
    & = \nu_1\big( (E'+n') \times \R \big) \\
    & = \nu_1\big( \pi^{-1}_{d,d-1}(E'+n') \big) \\
    & = \nu'(E'+n') =0.
  \end{align*}
  Since $\nu_2$ is concentrated on $[0,1]^{d-1} \times \{0, 1/m_d, 2/m_d, \ldots, 1 \}$, we have $\nu_2(E+n)=0$.
  Thus, we conclude that $\nu(E+n)=0$.
  Therefore, the set $E$ is desired.

  \textbf{Case II}: for each $1\le j \le d$, we have $\mu\big( \{ (x_1,\ldots,x_d) \in [0,1]^d: 0< x_j < 1 \} \big) =0$.
  Note that $\mathrm{spt}(\mu) \sse [0,1]^d$.
  Thus, we conclude that
  \begin{equation}\label{support-mu}
    \mathrm{spt}(\mu) \sse \{(x_1,x_2,\ldots,x_d) \in [0,1]^d: x_j \in \{0,1\}\}.
  \end{equation}

  If $\bf{0} \not\in \mathrm{spt}(\mu)$, then arbitrarily take $\tau=(\tau_1,\tau_2,\ldots, \tau_d) \in \mathrm{spt}(\mu)$.
  We have $\tau_j =0$ or $\tau_j =1$ for $1\le j \le d$.
  We define $\phi: \R^d \to \R^d$ by $(x_1, x_2,\ldots, x_d) \mapsto (y_1,y_2,\ldots,y_d)$ where $y_j = 1- x_j$ if $\tau_j =1$ and $y_j = x_j$ if $\tau_j =0$.
  Note that $\phi^{-1} = \phi$ and $\phi(\tau) = \bf{0}$.
  Set $\nu' = \nu \circ \phi$.
  Then we have \[ \nu' = (\lambda* \mu) \circ R \circ \phi= \big( (\lambda \circ \psi) * (\mu \circ \phi) \big) \circ R, \]
  where $\psi: \R^d \to \R^d$ is defined by $(x_1, x_2,\ldots, x_d) \mapsto (y_1,y_2,\ldots,y_d)$ where $y_j = m_j -1 - x_j$ if $\tau_j =1$ and $y_j = x_j$ if $\tau_j =0$.
  Write $\lambda' = \lambda \circ \psi$ and $\mu' = \mu \circ \phi$.
  It follows that \[ \nu' = (\lambda' * \mu') \circ R. \]
  Set $B' = \psi(B)$ and $p_{b'} = p_{\psi(b')}$ for $b' \in B'$.
  Then we have $\lambda' = \sum_{b' \in B'} p_{b'} \delta_{b'}$.
  Note that $(R,B') \in \mcal{D}_d$ and $\mathrm{spt}(\mu') = \phi\big(\mathrm{spt}(\mu)\big) \sse [0,1]^d$.

  Assume that there exists a Borel subset $E' \sse \R^d$ such that $\nu'(E')>0$ and $\nu'(E'+n') =0$ for all $n' \in \Z^{d} \setminus \{{\bf 0}\}$.
  Let $E=\phi(E')$. Then $\nu(E) = \nu\big( \phi(E')\big) = \nu'(E') >0$.
  For $n =(n_1, n_2, \ldots, n_d)\in \Z^{d} \setminus \{{\bf 0}\}$, let $n'=(n'_1, n'_2, \ldots , n'_d)$ where $n'_j = - n_j$ if $\tau_j =1$ and $n'_j = n_j$ if $\tau_j =0$. Then we have
  \[ \nu(E + n) = \nu\big( \phi(E') + n\big) =\nu\big( \phi(E' +n')\big) = \nu'(E'+n') =0. \]
  The set $E$ is desired.
  Thus, it suffices to show the conclusion holds for $\nu'$.
  Note that $\bf{0} =\phi(\tau) \in \phi\big( \mathrm{spt}(\mu) \big) = \mathrm{spt}(\mu')$.
  Therefore, without loss of generality, we assume that ${\bf 0}\in \mathrm{spt}(\mu)$.
  Thus, by (\ref{support-nu}) we have \[ R^{-1} B \sse \mathrm{spt}(\nu). \]

  In the following, we show that there exists $b\in  B$ such that
  \begin{equation}\label{li1}
    R^{-1}b+n\notin \mathrm{spt}(\nu)\;\; \text{for any}\;\; n \in \Z^{d} \setminus \{{\bf 0}\}.
  \end{equation}
  Then the set $E=\{R^{-1}b\}$ is desired.

  Now arbitrarily take $b^{(1)}=(b^{(1)}_1, b^{(1)}_2, \ldots , b^{(1)}_d)\in   B$. If (\ref{li1}) holds, then we are done. Otherwise, there exists $n^{(1)}=(n^{(1)}_1, n^{(1)}_2, \ldots , n^{(1)}_d) \in \Z^{d} \setminus \{ {\bf 0} \}$ such that
  \[ R^{-1}b^{(1)} + n^{(1)}\in \mathrm{spt}(\nu). \]
  Note that $\mathrm{spt}(\nu) \sse [0,1]^d$. Thus, there exists $1\leq i_1\leq d$ such that
  \begin{equation}\label{li2}
  n^{(1)}_{i_1}=1\;\;\text{and}\;\; b^{(1)}_{i_i}=0.
  \end{equation}
  On the other hand, by \eqref{support-nu}, we can find $b^{(2)}=(b^{(2)}_1, b^{(2)}_2, \ldots , b^{(2)}_d)\in B$ and $y^{(1)} \in \mathrm{spt}(\mu)$ such that
  \begin{equation}\label{li3}
    R^{-1}b^{(1)} + n^{(1)} = R^{-1} b^{(2)} + R^{-1} y^{(1)}.
  \end{equation}
  Combining (\ref{li2}) and (\ref{li3}), one has $b^{(2)}_{i_1}=m_{i_1}-1$.

  If $b^{(2)}$ satisfies (\ref{li1}), then we are done. Otherwise, by \eqref{support-nu}, one can find $n^{(2)}=(n^{(2)}_1, n^{(2)}_2, \ldots , n^{(2)}_d) \in \Z^{d} \setminus \{ {\bf 0} \}$, $b^{(3)}=(b^{(3)}_1, b^{(3)}_2, \ldots , b^{(3)}_d)\in B$ and $y^{(2)} \in \mathrm{spt}(\mu)$ such that
  \begin{equation}\label{li4}
    R^{-1}b^{(2)} + n^{(2)} = R^{-1} b^{(3)} + R^{-1} y^{(2)}.
  \end{equation}
  By the above argument, there exists $i_2\in \{1,2,\ldots, d\}\setminus \{i_1\}$
  such that $n^{(2)}_{i_2}=1$ and $b^{(2)}_{i_2}=0$.
  Note that $b^{(2)}_{i_1}=m_{i_1}-1$, $b^{(2)}_{i_2}=0$ and $n^{(2)}_{i_2}=1$ . Thus (\ref{li4}) implies the following facts
  $$
  b^{(3)}_{i_1}\geq m_{i_1}-2,\; b^{(3)}_{i_2}=m_{i_2}-1.
  $$

  If $b^{(3)}$ satisfies (\ref{li1}), then we are done. Otherwise, one can find $b^{(4)}=(b^{(4)}_1, b^{(4)}_2, \ldots , b^{(4)}_d)\in B$ satisfying
  $$
  b^{(4)}_{i_1}\geq m_{i_1}-3,\; b^{(4)}_{i_2}\geq m_{i_2}-2,\;
  b^{(4)}_{i_3}=m_{i_3}-1\;\;\text{where}\;\; i_3\in \{1,2,\ldots, d\}\setminus \{i_1,i_2\}.
  $$
  We can continue the above process at most $d$ steps to obtain $b \in B$ which satisfies (\ref{li1}).
  This is because all $m_i\geq d+1$ and so at the $d$-th step, one can find $b^{(d+1)}=(b^{(d+1)}_1, b^{(d+1)}_2, \ldots , b^{(d+1)}_d)\in B$ satisfying $b^{(d+1)}_i\geq 1$ for all $1\leq i\leq d$. This $b^{(d+1)}$ clearly satisfies (\ref{li1}).

  Now, by assuming that Theorem \ref{key-theorem} holds in $\R^{d-1}$, we have showed that Theorem \ref{key-theorem} holds in $\R^d$.
  By Lemma \ref{lemma-dimension-1}, Theorem \ref{key-theorem} holds in $\R$.
  By induction on the dimension $d$, we complete the proof.
\end{proof}

We end this section by giving a corollary of Theorem \ref{key-theorem}.

\begin{corollary}\label{coro-infinite-convolution}
  In $\R^d$, given a sequence $\{(R_n, B_n)\}_{n=1}^\f \subseteq \mcal{D}_d$, for the infinite convolution
  \[ \nu = \delta_{R_1^{-1} B_1} * \delta_{(R_2 R_1)^{-1} B_2} * \cdots * \delta_{(R_n \cdots R_2 R_1)^{-1} B_n} * \cdots, \]
  we have $\mcal{Z}(\nu) = \emptyset$.
\end{corollary}
\begin{proof}
  Let $R=R_d \cdots R_2 R_1$, and \[ B = R_d R_{d-1} \cdots R_2 B_1 + R_d R_{d-1}\cdots R_3 B_2 + \cdots + R_d B_{d-1} + B_d.\]
  Then $(R,B) \in \mcal{D}_d$, and all diagonal elements of $R$ are $\ge 2^d \ge d+1$.
  The measure $\nu$ can be written as \[ \nu= (\delta_B * \mu)\circ R, \]
  where $\mu = \delta_{R_{d+1}^{-1} B_{d+1}} * \delta_{(R_{d+2} R_{d+1})^{-1} B_{d+2}} * \cdots$.
  Note that $\mathrm{spt}(\mu) \sse [0,1]^d$.
  It follows from Theorem \ref{key-theorem} and Theorem \ref{theorem-emptyset} that $\mcal{Z}(\nu) = \emptyset$.
\end{proof}

\section{Proofs of Theorem \ref{theorem-general-spectrality}, \ref{theorem-Rd-1} and \ref{theorem-Rd-2}} \label{sec_Rd}

First, we recall a criterion for the convergence of infinite convolutions, and then we prove Theorem \ref{theorem-general-spectrality}.

\begin{theorem}\label{theorem-weak-convergence}{\rm\cite[Theorem 1]{Jessen-Wintner-1935}}
  Let $\mu_1,\mu_2, \cdots \in \mcal{P}(\R^d)$.
  A necessary and sufficient condition for the convergence of the infinite convolution $\mu_1* \mu_2* \cdots$ is that for any arbitrarily chosen sequence $\{k_n\}$ of positive integers, $\{\rho_n\}$ converges weakly to $\delta_{{\bf 0}}$ as $n \to \f$, where $\rho_n = \mu_{n+1} * \cdots * \mu_{n+k_n}$.
\end{theorem}

\begin{theorem}\label{theorem-general-spectrality-2}
  Let $\{(R_n,B_n)\}_{n=1}^\infty$ be a sequence of admissible pairs in $\R^d$.
  Suppose that

  {\rm(i)} \[\lim_{n \to \f} \|R_1^{-1} R_2^{-1} \cdots R_n^{-1}\| = 0,\]

  \noindent
  and there exists a cube $C= t_0 + [0,1]^d$ for some $t_0 \in \R^d$ such that

  {\rm(ii)} for each $n \ge 1$, we have $R_n^{-1}(C + b) \subseteq C$ for every $b \in B_n$,

  {\rm(iii)} there exists an admissible pair $(R,B)$, which occurs infinitely times in the sequence $\{(R_n,B_n)\}_{n=1}^\f$, such that $R^{-1}(C + b_0) \subseteq \mathrm{int}(C)$ for some $b_0 \in B$.

  \noindent
  Then the infinite convolution $\mu$ defined in \eqref{infinite-convolution} exists and is a spectral measure with a spectrum in $\Z^d$.
\end{theorem}
\begin{proof}
  Let $\{\mu_n\}$ be defined in \eqref{ms_mun}. We have
  \[ \mathrm{spt}(\mu_n) = R_1^{-1} B_1 + (R_2 R_1)^{-1} B_2 + \cdots + (R_n \cdots R_2 R_1)^{-1} B_n. \]
  By the assumption (ii), we have $R_n^{-1}(B_n+C) \sse C$ for all $n \ge 1$.
  Thus,
  \begin{align*}
    &~R_1^{-1} B_1 + (R_2 R_1)^{-1} B_2 + \cdots + (R_n \cdots R_2 R_1)^{-1} B_n + (R_n \cdots R_2 R_1)^{-1}C \\
    = & ~ R_1^{-1} B_1 + (R_2 R_1)^{-1} B_2 + \cdots + (R_{n-1} \cdots R_2 R_1)^{-1} \Big( R_n^{-1}(B_n +C) \Big) \\
    \subseteq & ~ R_1^{-1} B_1 + (R_2 R_1)^{-1} B_2 + \cdots + (R_{n-1} \cdots R_2 R_1)^{-1} B_{n-1} +  (R_{n-1} \cdots R_2 R_1)^{-1}  C \\
    \subseteq & ~ \cdots\cdots \\
    \subseteq & ~ R_1^{-1} (B_1 + C) \\
    \subseteq & ~ C.
  \end{align*}
  It follows that
  \begin{equation}\label{support-mun}
    \mathrm{spt}(\mu_n) \subseteq C - (R_n \cdots R_2 R_1)^{-1} C.
  \end{equation}

  Choose a sequence $\{k_n\}_{n=1}^\f$ of positive integers in an arbitrary way.
  For each $n \ge 1$, define \[ \rho_n = \delta_{(R_{n+1} \cdots R_2 R_1)^{-1} B_{n+1}} * \cdots * \delta_{(R_{n+k_n} \cdots R_2 R_1)^{-1} B_{n+k_n}}. \]
  Then $\rho_n = \lambda_n \circ (R_n \cdots R_2 R_1)$ where $\lambda_n = \delta_{R_{n+1}^{-1} B_{n+1}} * \cdots * \delta_{(R_{n+k_n} \cdots R_{n+2} R_{n+1})^{-1} B_{n+k_n}}$.
  By the same argument for $\mu_n$, we have $\mathrm{spt}(\lambda_n) \subseteq C - (R_{n+k_n} \cdots R_{n+2} R_{n+1})^{-1}C$.
  Thus, \[ \mathrm{spt}(\rho_n) = (R_n \cdots R_2 R_1)^{-1} \mathrm{spt}(\lambda_n) \subseteq (R_n \cdots R_2 R_1)^{-1} C - (R_{n+k_n} \cdots R_{2} R_{1})^{-1}C. \]
  By the assumption (i), the sequence of sets $\{\mathrm{spt}(\rho_n)\}$ shrinks to the origin point ${\bf 0}$.
  Therefore, we have $\{\rho_n\}$ converges weakly to $\delta_{{\bf 0}}$.
  Since the sequence $\{k_n\}_{n=1}^\f$ is arbitrarily chosen, by Theorem \ref{theorem-weak-convergence}, we conclude that the infinite convolution $\mu$ exists.

  For $\gamma >0$, define $C_\gamma= C + [-\gamma, \gamma]^d$. It follows from \eqref{support-mun} and the assumption (i) that for any given $\gamma >0$, $\mathrm{spt}(\mu_n) \subseteq C_\gamma$ for sufficiently large $n$. Note that $\{\mu_n\}$ converges weakly to $\mu$. Thus, we have $\mathrm{spt}(\mu) \subseteq C_\gamma$ for all $\gamma > 0$. It follows that $\mathrm{spt}(\mu) \subseteq C$.
  Let $\{\nu_{>n}\}$ be defined in \eqref{ms_nugen}.
  Similarly, we also have $\mathrm{spt}(\nu_{>n}) \subseteq C$ for all $n \ge 1$.
  Since $(R,B)$ occurs infinitely times in the sequence $\{(R_n,B_n)\}_{n=1}^\f$, let
  \[ \{n_1 , n_2 , n_3, \ldots \} = \{ n \ge 1: (R_n, B_n) = (R,B) \}, \] where $n_1 < n_2 < n_3 <\cdots$.
  By the weak compactness of $\{\nu_{>n}\}$, the sequence $\{\nu_{>n_k}\}$ has a weak convergent subsequence.
  By taking a subsequence, we may assume that $\{\nu_{>n_k}\}$ converges weakly to a Borel probability measure $\rho$. Moreover, we have $\mathrm{spt}(\rho) \sse C$.
  Note that $\nu_{> n_k -1} = \delta_{R_{n_k}^{-1} B_{n_k}} * (\nu_{>n_k} \circ R_{n_k}) = \delta_{R^{-1} B} * (\nu_{>n_k} \circ R)$.
  Thus, $\{\nu_{>n_k-1}\}$ converges weakly to $\nu=\delta_{R^{-1} B} * (\rho \circ R)$.
  It follows that \[\mathrm{spt}(\nu) = R^{-1}\big( \mathrm{spt}(\rho) + B \big). \]
  By the assumption (ii), we have $\mathrm{spt}(\nu) \sse C$.
  By the assumption (iii), we have $E=\mathrm{spt}(\nu) \cap \mathrm{int}(C) \ne \emptyset$, and hence $\nu(E)>0$.
  For any $n \in \Z^d\sm \{{\bf 0}\}$, noting that $E+n \subset \mathrm{int}(C) + n$, we have $(E+n) \cap C = \emptyset$. It follows that $\nu(E+n) = 0$ for any $n \in \Z^d\sm \{{\bf 0}\}$.
  By Theorem \ref{theorem-emptyset}, we have $\mcal{Z}(\nu)=\emptyset$.
  It follows from Theorem \ref{theorem-spectrality-emptyset} that $\mu$ is a spectral measure with a spectrum in $\Z^d$.
\end{proof}
\begin{remark}
  (a) By Lemma \ref{lemma-HT} (i) and Lemma \ref{lemma-affine-invariant}, the translation of the digit sets $B_n$ does not change the admissible pair assumption and the spectrality of the resulting infinite convolution.
  Thus, in practical applications we can translate the digit sets $B_n$ to check whether the assumptions (ii) and (iii) are satisfied or not.

  (b) If there are only finitely many terms in the sequence $\{(R_n,B_n)\}_{n=1}^\infty$ not satisfying the assumption (ii), then we can choose a sufficiently large integer $n_0$ such that the admissible pair $(R_n,B_n)$ satisfies the condition (ii) for all $n \ge n_0+1$.
  Applying this theorem to the sequence $\{ (R_n,B_n) \}_{n=n_0+1}^\f$, we obtain that $\nu_{>n_0}$ defined in \eqref{ms_nugen} is a spectral measure with a spectrum in $\Z^d$.
  Note that $\mu_{>n_0} = \nu_{>n_0} \circ (R_{n_0} \cdots R_2 R_1)$.
  By Lemma \ref{lemma-affine-invariant}, $\mu_{>n_0}$ admits a spectrum in $(R_{n_0} \cdots R_2 R_1)^{\mathrm{T}}\Z^d$.
  Finally, by Lemma \ref{lemma-spectrality-convolution}, we also conclude the infinite convolution $\mu = \mu_{n_0} * \mu_{>n_0}$ is a spectral measure.

  (c) If we replace the set $C$ by a smaller set $C'$ which is contained in some cube $t_0 + [0,1]^d$ in the statement, then the result remains valid.
\end{remark}

By using Theorem \ref{key-theorem}, the proof of Theorem \ref{theorem-Rd-1} is similar to that of Theorem \ref{theorem-general-spectrality}.

\begin{theorem}\label{theorem-Rd-1-2}
  In $\R^d$, suppose that $\{(R_n, B_n)\}_{n=1}^\f \subseteq \mcal{D}_d$ is a sequence of admissible pairs.
  If there exists an admissible pair $(R,B)$ that occurs infinitely times in the sequence $\{(R_n,B_n)\}_{n=1}^\f$, and moreover, all diagonal elements of $R$ are greater than or equal to $d+1$,
  then the infinite convolution $\mu$ defined in \eqref{infinite-convolution} is a spectral measure with a spectrum in $\Z^d$.
\end{theorem}
\begin{proof}
  Let \[ \{n_1, n_2, n_3, \ldots \} = \{n \ge 1: (R_{n},B_{n})=(R,B) \}, \] where $n_1 < n_2 < n_3 < \cdots$.
  Let $\{\nu_{>n}\}$ be defined in \eqref{ms_nugen}. Since $\{(R_n, B_n)\}_{n=1}^\f \subseteq \mcal{D}_d$, we have $\mathrm{spt}(\nu_{>n}) \sse [0,1]^d$ for all $n \ge 1$.
  By taking the subsequence of $\{n_k\}$, we may assume that $\{ \nu_{>n_k} \}$ converges weakly to a Borel probability measure $\rho$ with $\mathrm{spt}(\rho) \sse [0,1]^d$.
  Similarly, note that $\nu_{> n_k -1} = \delta_{R_{n_k}^{-1} B_{n_k}} * (\nu_{>n_k} \circ R_{n_k}) = \delta_{R^{-1} B} * (\nu_{>n_k} \circ R)$, and thus we have $\{\nu_{>n_k -1}\}$ converges weakly to $\nu = \delta_{R^{-1} B} * (\rho \circ R) = (\delta_B * \rho) \circ R$.
  By Theorem \ref{key-theorem} and Theorem \ref{theorem-emptyset}, we have $\mcal{Z}(\nu) =\emptyset$.
  By Theorem \ref{theorem-spectrality-emptyset}, $\mu$ is a spectral measure with a spectrum in $\Z^d$.
\end{proof}

The following lemma is needed to prove Theorem \ref{theorem-Rd-2}.
\begin{lemma}\label{lemma-uniform-approximation}
  For $f \in C_b(\R^d)$ and $\ep>0$, there exists $n_0 \ge 1$ such that for any sequence $\{(R_n,B_n)\}_{n=1}^\f \subseteq \mcal{D}_d$, we have
  \begin{equation*}
    \bigg| \int_{\R^d} f(x) \D\mu_{n_0}(x) - \int_{\R^d} f(x) \D\mu(x) \bigg| < \ep,
  \end{equation*}
  where $\mu$ is the infinite convolution defined in \eqref{infinite-convolution} and $\{\mu_n\}$ is defined in \eqref{ms_mun}.
\end{lemma}
\begin{proof}
  Let $\{\nu_{>n}\}$ be defined in \eqref{ms_nugen}. Since $\{(R_n,B_n)\}_{n=1}^\f \subseteq \mcal{D}_d$, for each $n \ge 1$ we have \[ \mathrm{spt}(\mu_{n}) \sse[0,1]^d \text{ and } \mathrm{spt}( \nu_{>n} ) \sse [0 , 1]^d. \] Note that $\nu_{>n} = \mu_{>n} \circ (R_n \cdots R_2 R_1)^{-1}$. This implies that \[ \mathrm{spt}(\mu_{>n}) = (R_n \cdots R_2 R_1)^{-1} \mathrm{spt}(\nu_{>n}) \sse [0,2^{-n}]^d. \]

  For $f \in C_b(\R^d)$ and $\ep>0$, since $f$ is uniformly continuous on $[0,2]^d$, there exists $0< \gamma<1$ such that for all $x,y\in [0,2]^d$ with $|x-y| < \gamma$ we have
  \begin{equation}\label{uniform-continuous}
    |f(x)-f(y)|< \ep.
  \end{equation}
  Since
  \begin{align*}
    \int_{\R^d} f(x) \D \mu(x)
    & = \int_{\R^d} f(x) \D \mu_{n} * \mu_{>n}(x) \\
    & = \int_{\R^d} \int_{\R^d} f(x+y) \D \mu_{>n}(y) \D \mu_{n}(x),
  \end{align*}
  by choosing a sufficiently large integer $n_0$ such that $2^{-n_0} < \gamma/\sqrt{d}$ and  \eqref{uniform-continuous}, we have that
  \begin{align*}
    &\ \bigg| \int_{\R^d} f(x) \D\mu_{n_0}(x) - \int_{\R^d} f(x) \D\mu(x) \bigg| \\
    = &\ \bigg| \int_{\R^d} \int_{\R^d} \big( f(x) - f(x+y) \big) \D \mu_{>n_0}(y) \D \mu_{n_0}(x) \bigg| \\
   \le &\ \int_{[0,1]^d} \int_{ [0 , 2^{-n_0} ]^d} \big|  f(x) - f(x+y) \big| \D \mu_{>n_0}(y) \D \mu_{n_0}(x)   \\
    < &\ \ep,
  \end{align*}
  as desired.
\end{proof}

\begin{theorem}\label{theorem-Rd-2-2}
  In $\R^d$, given a sequence of admissible pairs $\{(R_n, B_n)\}_{n=1}^\f \subseteq \mcal{D}_d$, if
  \begin{equation}\label{finiteness}
    \sup_{n \ge 1} \|R_n\| < \f,
  \end{equation}
  then the infinite convolution $\mu$ defined in \eqref{infinite-convolution} is a spectral measure with a spectrum in $\Z^d$.
\end{theorem}
\begin{proof}
  Let $\mathcal{H} = \{ (R_n, B_n): n \ge 1 \}$.
  The assumption \eqref{finiteness} implies that the set $\mathcal{H}$ is finite, denoted by
  $$ \mathcal{H} = \{ (R'_1, B'_1), (R'_2, B'_2), \ldots, (R'_m, B'_m) \}. $$
  Let $\Omega = \{1,2,\ldots, m\}^\N$.
  For $\omega = (\omega_k)_{k=1}^\f \in \Omega$, let
  \[ \mu_\omega = \delta_{ (R'_{\omega_1})^{-1} B'_{\omega_1} } * \delta_{ (R'_{\omega_2}R'_{\omega_1})^{-1} B'_{\omega_2} } * \cdots * \delta_{ (R'_{\omega_k} \cdots R'_{\omega_2} R'_{\omega_1})^{-1} B'_{\omega_k} } * \cdots. \]
  For $q \ge 1$, we write
  \[ \mu_{\omega,q} = \delta_{ R_{\omega_1}^{-1} B_{\omega_1} } * \delta_{ (R_{\omega_2} R_{\omega_1})^{-1} B_{\omega_2} } * \cdots * \delta_{ (R_{\omega_q} \cdots R_{\omega_2} R_{\omega_1})^{-1} B_{\omega_q} }. \]

  For the infinite convolution $\mu$ defined in \eqref{infinite-convolution}, there exists $\eta \in \Omega$ such that $\mu = \mu_\eta$.
  Let $\{\nu_{>n}\}$ be defined in \eqref{ms_nugen}, and let $\sigma$ denote the left shift on $\Omega$.
  Then we have \[ \nu_{>n} = \mu_{\sigma^n(\eta)}. \]
  By the compactness of $\Omega$, there exists a subsequence $\{n_j\}$ of positive integers such that $\{\sigma^{n_j}(\eta)\}_{j=1}^\f$ converges to $\zeta$ in $\Omega$ for some $\zeta \in \Omega$.

  For $f \in C_b(\R^d)$ and $\ep >0$, by Lemma \ref{lemma-uniform-approximation}, there exists $q_0 \ge 1$ such that for all $j \ge 1$,
  $$ \bigg| \int_{\R^d} f(x) \D \mu_{ \sigma^{n_j}(\eta), q_0 }(x) -  \int_{\R^d} f(x) \D \mu_{ \sigma^{n_j}(\eta)} (x) \bigg| < \frac{\ep}{2}, $$
  and
  $$ \bigg| \int_{\R^d} f(x) \D \mu_{\zeta, q_0}(x) -  \int_{\R^d} f(x) \D \mu_{\zeta}(x) \bigg| < \frac{\ep}{2}. $$
  Since $\set{  \sigma^{n_j}(\eta) }_{j=1}^\f$ converges to $\zeta$,
  there exists $j_0 \ge 1$ such that for $j \ge j_0$, we have
  $$ \mu_{ \sigma^{n_j}(\eta), q_0 } = \mu_{\zeta, q_0}. $$
  Thus, it follows that for $j \ge j_0$,
  $$ \bigg| \int_{\R^d} f(x) \D \mu_{ \sigma^{n_j}(\eta)}(x) - \int_{\R^d} f(x) \D \mu_{\zeta}(x) \bigg| <\ep. $$
  Since $f$ and $\ep$ are arbitrary, this implies that the sequence $\{\nu_{>n_j} = \mu_{ \sigma^{n_j}(\eta)} \}$ converges weakly to $\mu_{ \zeta}$.

  By Corollary \ref{coro-infinite-convolution}, we have $\mathcal{Z}(\mu_\zeta) = \emptyset$.
  By Theorem \ref{theorem-spectrality-emptyset}, it follows that $\mu$ is a spectral measure with a spectrum in $\Z^d$.
\end{proof}

\section*{Acknowledgements}

Wenxia Li is supported by NSFC No. 12071148, 11971079 and Science and Technology Commission of Shanghai Municipality (STCSM)  No.~22DZ2229014..
Zhiqiang Wang is supported by Fundamental Research Funds for the Central Universities No.~YBNLTS2023-016.

\end{document}